\documentclass[oneside]{amsart}

\usepackage[letterpaper,body={14.6cm,22.5cm}, mag=1000]{geometry}
\usepackage{amssymb}
\usepackage{amsthm}
\usepackage{amscd}
\usepackage{mathtext}
\usepackage[T2C]{fontenc}
\usepackage[koi8-r]{inputenc}
\usepackage[russian,english]{babel}

\numberwithin{equation}{section}
\theoremstyle{plain}

\newtheorem{thm}{Theorem}[section]
 
 \newtheorem{lemma}[thm]{Lemma}

\newtheorem*{thma}{Theorem A}
\newtheorem*{thmb}{Theorem B}

\theoremstyle{definition}

\newcommand{\dlabel}[1]{\ifmmode \text{\ttfamily \upshape [#1] } \else
{\ttfamily \upshape [#1] }\fi \label{#1}}

\newcommand{\gen}[1]{\left < #1 \right >}

\DeclareFontFamily{U}{wncy}{}
   \DeclareFontShape{U}{wncy}{m}{n}{<->wncyr10}{}
   \DeclareSymbolFont{mcy}{U}{wncy}{m}{n}
   \DeclareMathSymbol{\Sh}{\mathord}{mcy}{"58}

\begin{document}

\title{On $\Sh$-rigidity of groups of order $p^6$}

\author{Pradeep K. Rai}
\address{School of Mathematics, Harish-Chandra Research Institute, Chhatnag Road, Jhunsi, Allahabad 211019, INDIA.}
\email{pradeeprai@hri.res.in}

\author{Manoj K. Yadav}
\address{School of Mathematics, Harish-Chandra Research Institute, Chhatnag Road, Jhunsi, Allahabad 211019, INDIA.}
\email{myadav@hri.res.in}


\begin{abstract}
Let $G$ be a group and $Out_c(G)$ be the group of its class-preserving outer automorphisms. We compute $|Out_c(G)|$ for all the group $G$ of order $p^6$, where $p$ is an odd prime. 
As an application, we observe that if $G$ is a $\Sh$-rigid group of order $p^6$, then it's Bogomolov multiplier $B_0(G)$ is zero. 
\end{abstract}

\maketitle

\section{Introduction}

\noindent Let $G$ be a group which acts on itself by conjugation, and let $H^1(G, G)$ be the first cohomology pointed set. Denote by $\Sh(G)$ the subset of $H^1(G, G)$
consisting of the cohomology classes becoming trivial after restricting to every cyclic subgroup of G. The set $\Sh(G)$, for a given group $G$, is called the Shafarevich-Tate set of $G$.
Following Kunyavski\u{\i} \cite{BK}, we say that $G$ is a $\Sh$-rigid  group if the set $\Sh(G)$ consists of one element. Throughout the paper $\Sh$-rigid  groups will be called rigid groups. 
The Bogomolov multiplier $B_0(G)$ of a finite group G is defined as the subgroup of the Schur multiplier consisting of the cohomology classes vanishing after
restriction to all abelian subgroups of G. See \cite{BK} for further details on rigidity and  Bogomolov multiplier of certain groups, and eventual ties between them. Kang and Kunyavski\u{\i} in \cite{KK} observed that $B_0(G) = 0$ for most of the known classes of finite $\Sh$-rigid groups $G$, and asked the following question:
\vspace{.05in}

\noindent {\bf Question}(\cite[Question 3.2]{KK}). Let $G$ be a finite  $\Sh$-rigid group.  Is it true that $B_0(G) = 0$?

\vspace{.1in}

This question has an affirmative answer for all $p$-groups of order at most $p^5$. For $p$-groups of order $\le p^4$, it is proved in \cite{KK}, and for groups of order $p^5$, $p > 3$, it follows from \cite{PM} and \cite{MKY}. For $p = 2, 3$, it can be easily verified using GAP \cite{GAP}. 
However this question does not have a positive solution in general (negative answer for a group of order $256$ is provided in the paper \cite{KK} itself), it is interesting to explore for which classes of groups it has a positive solution. This paper is devoted to  studying the  rigidity property of groups of order $p^6$ for an odd prime $p$ and answering the above question in affirmative for the class of groups considered here. Our viewpoint on this study is a bit different. We study rigidity problem through automorphisms of groups. We make it more precise here. For a finite group $G$, an automorphism of $G$ is called (conjugacy) class-preserving if it maps each  element of $G$ to  some conjugate of it. The set of all class-preserving automorphisms of $G$, denoted by $Aut_c(G)$, forms a normal subgroup of $Aut(G)$ and contains $Inn(G)$, the group of all inner automorphisms of $G$. Let $Out_c(G)$ denote the quotient group $Aut_c(G)/Inn(G)$. Then there is a bijection between $Out_c(G)$ and $\Sh(G)$ for any finite group $G$ \cite[2.12]{Ono}. Therefore a finite group $G$ is $\Sh$-rigid if and only if $Out_c(G) = 1$.  Both approaches, i.e., through the Shafarevich-Tate sets (cf.  \cite{KV1, OW1, OW2}) as well as $Out_c(G)$ (cf. \cite{FtSt, HL, MV2, MV3, MKY}), have been explored by many mathematicians studying the rigidity of various classes of groups. As mentioned above, in the present paper we take the latter approach. 

Groups of order $p^6$, $p$ an odd prime, are classified in $43$ isoclinism families  by R. James  \cite{RJ}, which are denoted by $\Phi_k$ for $1 \le k \le 43$.  The concept of isoclinism was introduced by P. Hall \cite{PH}. That $Aut_c(G)$, for a non-abelian finite group $G$,  is independent (upto isomorphism) of the choice of a group in a given isoclinism family  is shown in \cite[Theorem 4.1]{MKY}. This result allows us to select and work with any group from each of $43$ isoclinism families of groups of order $p^6$. The following is the main result of this paper, which classifies rigid groups of order $p^6$, $p$ being an odd prime, and size of $| \Sh(G)| = |Out_c(G)|$ is computed for the groups $G$ which are not rigid.

\begin{thma}
 Let $G$ be a group of order $p^6$ for an odd prime $p$. Then $Out_c(G) \neq 1$ if and only if $G$ belongs to one of the isoclinism families 
 $\Phi_k$ for $k = 7, 10, 13, 15, 18, 20, 21, 24, 30, 36, 38, 39$. Moreover,
 \begin{enumerate}
 \item  if $G$ belongs to one of the isoclinism families $\Phi_k$ for $k = 7, 10, 24, 30, 36, 38, 39$, then $|Out_c(G)|$ = $p$,  
 \item  if $G$ belongs to one of the isoclinism families $\Phi_k$ for $k = 13, 18, 20$, then $|Out_c(G)|$ = $p^2$, and
 \item  if $G$ belongs to one of the isoclinism families $\Phi_k$ for $k = 15, 21$, then $|Out_c(G)| = p^4$.
 \end{enumerate}
\end{thma}

Proof of Theorem A follows from Theorem \ref{prop1} and Theorem \ref{prop2}. Using GAP \cite{GAP} it can be verified that the above question has an affirmative answer for groups of order $2^6$ and $3^6$ (Bogomolov multiplier of groups of order $2^6$ is also computed in \cite{CHKK}). With this information and results of Chen and Ma \cite[Theorem 1.4]{CM}, Theorem A provides an affirmative answer of the above mentioned question of Kang and Kunyavski\u{\i} for groups of order $p^6$ for all primes $p$ in the following result.

\begin{thmb}
 Let $G$ be a $\Sh$-rigid group of order $p^6$ for a prime $p$. Then its Bogomolov multiplier $B_0(G)$ is zero.
\end{thmb}

Theorem A is also interesting from other point of view.  In 1911 W. Burnside \cite{WB1} asked the following question: Does there exist
any finite group $G$ such that $G$ has a non-inner class-preserving automorphism? In 1913, he himself answered this question affirmatively by constructing a group $W$ of order $p^6$, $p$ an odd prime \cite{WB}. 
This group $W$ is of nipotency class 2 such that $|Aut_c(W)| = p^8$ and $|Out_c(W)| = p^4$. More groups with non-inner class-preserving automorphism were constructed in 
\cite{BM, HEIN, HW, IM, FS}. We refer the reader to \cite{MKYSRV} for a more comprehensive survey on the topic. So Theorem A above classifies all groups of order $p^6$, $p$ an odd prime, having non-inner class-preserving automorphisms. \\

We use the following notations. For a multiplicatively written group $G$ let $x,y,a \in G$. Then $[x,y]$ denotes the commutator $x^{-1}y^{-1}xy$ and $x^a$ denotes the conjugate of $x$ by $a$ i.e. $a^{-1}xa$. 
By $\gen{x}$ we denote the cyclic subgroup of $G$ generated by $x$. By $Z(G)$ and $Z_2(G)$ we denote the center and the second center of $G$ respectively.  For $x \in G$, $C_G(x)$ denotes the centralizer  of $x$ in $G$.  By $C_G(H)$ we denote the centralizer of $H$ in $G$, where $H$ is a subgroup of $G$. We write 
$\gamma_2(G)$ for the commutator subgroup of $G$. The group of all homomorphisms from a group $H$ to an abelian group $K$ is denoted by $Hom(H, K)$. For $x \in G, \ x^G$ denotes 
the $G$-conjugacy class of $x$ and $[x, G]$ denotes the set of all $[x,y]$, for $y \in G$. By $C_p$ we denote the cyclic group of order $p$. The subgroup generated by all elements of order $p$ in $Z(G)$ is denoted by $\Omega_1(Z(G))$. \\

 \section{ \bf{Preliminaries}}
 
 The following commutator identities are easy to prove.
 
  \begin{lemma}\label{lem1}
   Let $G$ be a group and $x,y,z \in G$. Then 
   \begin{enumerate}
    \item $[xy, z] = [x, z]^y[y, z]$, and
    \item $[z, xy] = [z, y][z, x]^y$.
   \end{enumerate}
\end{lemma}

An automorphism of a group $G$ is called central if it acts trivially on the central quotient group $G/Z(G)$, or equivalently, if it commutes with all the inner automorphisms of $G$.
The set of all central automorphisms of $G$ forms a normal subgroup of $Aut(G)$. We denote this subgroup by $Autcent(G)$.  The following lemma follows from \cite{AY}.
  
  \begin{lemma}\label{lem2}
   Let $G$ be a purely non-abelian finite $p$-group. Then $|Autcent(G)|$ = $|Hom(G/\gamma_2(G),$ $Z(G))|$.
  \end{lemma}

  \begin{lemma}[\cite{MKY}, Lemma 2.2]\label{lem3}
  Let $G$ be a finite $p$-group such that $Z(G) \leq [x, G]$ for all $x \in G-\gamma_2(G)$. Then $|Aut_c(G)| \geq |Autcent(G)||G/Z_2(G)|$.
  \end{lemma}

 \begin{lemma}[\cite{HW}, Proposition 14.4]\label{lem4}
 Let $G$ be a finite group and $H$ be an abelian normal subgroup of $G$ such that $G/H$ is cyclic. Then $Out_c(G) = 1$. 
\end{lemma}

\begin{lemma}[\cite{PKR}, Lemma 5] \label{lempkr}
 For a group $G$,  $C_{Aut_c(G)}(Inn(G)) = Z(Aut_c(G))$.
\end{lemma}

\begin{lemma}[\cite{MKY}, Lemma 2.6]\label{lem6}
 Let $G$ be a finite group. Let ${x_1, x_2, \ldots,x_d}$ be a minimal generating set for $G$. Then $|Aut_c(G)| \leq \prod\limits_{i=1}^d |{x_i}^G|$.
\end{lemma}

\begin{lemma} \label{lemay}
Let $G$ be a purely non-abelian $p$-group, minimally generated by $\alpha_1, \alpha_2, \ldots, \alpha_t$.  Suppose  that $G/\gamma_2(G)$ is elementary abelian. If $\beta_i \in \Omega_1(Z(G))$ for $i = 1, \ldots,t$, then the map $\delta :\{\alpha_1, \alpha_2, \ldots, \alpha_t\} \rightarrow G$, defined as $\delta(\alpha_i) = \alpha_i\beta_i$, extends to a central automorphism of $G$.
\end{lemma}
\begin{proof}
Note that the map $f_{\delta} : \{\alpha_i\gamma_2(G) | \ i = 1, \ldots,t\} \rightarrow \Omega_1(Z(G))$ defined as 
$\alpha_i\gamma_2(G) \rightarrow \beta_i$ for $i = 1, \ldots,t$, extends to a homomorphism from $G/\gamma_2(G)$ to $Z(G)$ because $G/\gamma_2(G)$ is elementary abelian. Now it follows from \cite[Theorem 1]{AY} that the map $g \rightarrow gf_{\delta}(g)$ for all $g \in G$ is a central automorphism of $G$. This proves the lemma.  \hfill $\Box$

\end{proof}
Let $G$ be a group minimally generated by $\alpha_1, \alpha_2, \ldots, \alpha_t$. Then note that any element of $G$ can be written as 
$\eta\alpha_1^{k_1}\alpha_2^{k_2}\cdots\alpha_t^{k_t}$ for some $\eta \in \gamma_2(G)$ and some $k_1, k_2, \ldots,k_t \in \mathbb{Z}$. Also note that any conjugate of an $\alpha_i$ can be written as 
$\alpha_i[\alpha_i, \eta\alpha_1^{k_1}\alpha_2^{k_2}\cdots\alpha_t^{k_t}]$ for some $\eta \in \gamma_2(G)$ and some $k_1, k_2, \ldots,k_t \in \mathbb{Z}$. It follows that an automorphism $\delta$
of $G$ is class-preserving if and only if for every $k_1, k_2, \ldots,k_t \in \mathbb{Z}$ and for all $\eta_1 \in \gamma_2(G)$, there exist  $l_1, l_2, \ldots,l_t \in \mathbb{Z}$ and $\eta_2 \in \gamma_2(G)$ 
(depending on $k_1, k_2, \ldots,k_t$ and $\eta_1$) such that 
\[\bigg[\eta_1\prod\limits_{i=1}^t \alpha_i^{k_i}, \ \ \eta_2\prod\limits_{i=1}^t \alpha_i^{l_i}\bigg] = 
                                                        \bigg(\eta_1\prod\limits_{i=1}^t \alpha_i^{k_i}\bigg)^{-1}\delta\bigg(\eta_1\prod\limits_{i=1}^t \alpha_i^{k_i}\bigg).\]

We will be using these facts and Lemma \ref{lem1} very frequently in the proofs without any further reference.

\section{Groups $G$ with trivial $Out_c(G)$}

In this section, we deal with those groups $G$ for which $Out_c(G)$ is trivial. The isoclinism family $\Phi_k$ of groups of order $p^6$ contains a group $\Phi_k{(1^6)}$ for certain values of $k$. These groups will be considered frequently throughout the paper.

\begin{lemma}
 Let $G$ be the group $\Phi_{11}(1^6)$. Then $Out_c(G) = 1$.
\end{lemma}

\begin{proof}
The group $G$ is a special $p$-group, minimally generated by $\alpha_1, \alpha_2$ and $\alpha_3$. The commutator subgroup $\gamma_2(G)$ is generated by $\beta_1 := [\alpha_2, \alpha_3], \beta_2 := [\alpha_3, \alpha_1], \beta_3 := [\alpha_1, \alpha_2]$.
The conjugates of $\alpha_1, \alpha_2$ and $\alpha_3$ are  $\alpha_1\beta_3^t\beta_2^s$, $\alpha_2\beta_3^t\beta_1^r$ and
 $\alpha_3\beta_2^s\beta_1^r$ respectively, where $r$, $s$ and $t$ vary over $\mathbb{Z}$. Since the exponent of $\gamma_2(G)$ is $p$, it follows that  
$|\alpha_i^G| = p^2$ for $i = 1,2,3$. Therefore by Lemma \ref{lem6}, $|Aut_c(G)| \leq p^6$.
 Define a map $\delta : \{\alpha_1, \alpha_2, \alpha_3\} \rightarrow G$ such that $\alpha_1 \mapsto \alpha_1\beta_3^{t_1}\beta_2^{s_1}, \alpha_2 \mapsto 
 \alpha_2\beta_3^{t_2}\beta_1^{r_2}$ and $\alpha_3 \mapsto  \alpha_3\beta_2^{s_3}\beta_1^{r_3}$, for some $s_1, t_1, r_2, t_2, r_3, s_3 \in \mathbb{Z}$. 
 By Lemma \ref{lemay}, this map extends to a central automorphism of $G$. Since $\delta$ fixes $\gamma_2(G)$ element-wise,
 for $k_1, l_1, m_1 \in \mathbb{Z}$ and $\eta \in  \gamma_2(G)$, 
\[\delta(\eta\alpha_1^{k_1}\alpha_2^{l_1} \alpha_3^{m_1}) = 
 \eta\alpha_1^{k_1}\alpha_2^{l_1} \alpha_3^{m_1}\beta_1^{l_1r_2 + m_1r_3}\beta_2^{k_1s_1 + m_1s_3}\beta_3^{k_1t_1 + l_1t_2}.\]
 Therefore $\delta$ extends to a class-preserving automorphism if and only if for every $k_1, l_1, m_1 \in \mathbb{Z}$, and $\eta_1 \in \gamma_2(G)$ 
 there exist $k_2, l_2, m_2$ (depending on $k_1, l_1, m_1$) and $\eta_2 \in \gamma_2(G)$ such that
 \[[\eta_1\alpha_1^{k_1}\alpha_2^{l_1} \alpha_3^{m_1}, \eta_2\alpha_1^{k_2}\alpha_2^{l_2} \alpha_3^{m_2}] = 
   \beta_1^{l_1r_2 + m_1r_3}\beta_2^{k_1s_1 + m_1s_3}\beta_3^{k_1t_1 + l_1t_2}.\]
 Expanding the left hand side, we get
\[\beta_1^{l_1m_2 - l_2m_1}\beta_2^{k_2m_1 - k_1m_2}\beta_3^{k_1l_2 - k_2l_1} =  \beta_1^{l_1r_2 + m_1r_3}\beta_2^{k_1s_1 + m_1s_3}\beta_3^{k_1t_1 + l_1t_2}.\]
Comparing the powers of $\beta_i$'s, we have that  $\delta$ extends to a class-preserving automorphism if and only if the following equations hold true:
\[l_1m_2 - l_2m_1 \equiv l_1r_2 + m_1r_3 \pmod{p}\]
\[k_2m_1 - k_1m_2 \equiv k_1s_1 + m_1s_3 \pmod{p}\]
\[k_1l_2 - k_2l_1 \equiv k_1t_1 + l_1t_2 \pmod{p}. \]
Let $\delta$ be a class-preserving automorphism.
Choose $k_1 = 0$ and $m_1, l_1$ to be non-zero modulo $p$. Then $k_2m_1 \equiv m_1s_3 \pmod{p}$ and $-k_2l_1 \equiv l_1t_2 \pmod{p}$. It follows that $t_2 \equiv -s_3 \pmod{p}$. Similarly if we choose $l_1$ to be zero and 
$k_1, m_1$ to be non-zero modulo $p$, then $r_3 \equiv -t_1 \pmod{p}$, and if $m_1$ to be zero and $l_1, k_1$ to be non-zero modulo $p$, then $r_2 \equiv -s_1 \pmod{p}$.  
It follows that $|Aut_c(G)| \leq p^3 = |Inn(G)|$. Therefore $Out_c(G) = 1$.   \hfill $\Box$

\end{proof}

\begin{lemma}
 Let $G$ be one of the groups $\Phi_{17}(1^6)$ and $\Phi_{19}(1^6)$. 
 Then $Out_c(G) = 1$.
\end{lemma}

\begin{proof}
 First assume that $G$ is the group $\Phi_{17}(1^6)$. Then $G$ is a $p$-group of nilpotency class $3$, minimally generated by $\alpha, \alpha_1$ and $\beta$. The commutator subgroup 
 $\gamma_2(G)$ is abelian and generated by $\alpha_2 := [\alpha_1, \alpha], \alpha_3 := [\alpha_2, \alpha]$ and $\gamma := [\beta, \alpha_1]$. The center $Z(G)$ is of order $p^2$, generated by $\alpha_3$ and $\gamma$. It is easy to see that
 $|\alpha_1^{G}| \leq p^2, |\alpha^{G}| \leq p^2$
 and $|\beta^{G}| = p$. Therefore, applying Lemma \ref{lem6}, we have $|Aut_c(G)| \leq p^5$. Define a map $\delta : \{\alpha, \alpha_1, \beta\} \rightarrow G$ such that
 $\alpha \mapsto \alpha, \alpha_1 \mapsto \alpha_1$ and $\beta \mapsto \alpha_1^{-1}\beta\alpha_1$ = $\beta\gamma$. Suppose that $|Aut_c(G)| = p^5$. Then $\delta$
 must extend to a class-preserving automorphism of $G$. Hence there exist  $\eta_1 (= \alpha_2^{r_1}\alpha_3^{s_1}\gamma^{t_1}$ say) 
 $\in \gamma_2(G)$ and $k_1, l_1, m_1 \in \mathbb{Z}$ such that 
 $[\alpha\beta, \ \eta_1\alpha^{k_1}\alpha_1^{l_1}\beta^{m_1}] = \gamma$. It is a routine calculation to show that
\[[\alpha\beta, \eta_1\alpha^{k_1}\alpha_1^{l_1}\beta^{m_1}] = \alpha_2^{-l_1}\alpha_3^{-r_1}\gamma^{l_1},\]
which can not be equal to $\gamma$ for any value of $l_1$ and $r_1$. Hence $\delta$ can not be a  class-preserving automorphism, and hence  $|Aut_c(G)| \leq p^4$. But 
$|Inn(G)| = p^4$, so we have $Aut_c(G) = Inn(G)$.  

Now we take $G$ to be $\Phi_{19}(1^6)$. Then $G$ is  a $p$-group of class $3$, minimally generated by $\alpha, \alpha_1$ and $\alpha_2$. Define a map $\delta : \{\alpha, \alpha_1, \alpha_2\} \rightarrow G$ such that
 $\alpha \mapsto \alpha_1^{-1}\alpha\alpha_1 (=  \alpha\beta_1), \alpha_1 \mapsto \alpha_1$ and $\alpha_2 \mapsto \alpha_2$. Now the proof follows on the same lines as for the group $\Phi_{17}(1^6)$.
\hfill $\Box$

\end{proof}

\begin{lemma}
 Let $G$ be the group $\Phi_{23}(1^6)$. Then $Out_c(G) = 1$.
\end{lemma}

\begin{proof}
 The group $G$ is a  $p$-group of class $4$, minimally generated by $\alpha, \alpha_1$. The commutator subgroup $\gamma_2(G)$ is abelian and generated by $\alpha_{i+1} := [\alpha_i, \alpha]$ for $1 \leq i \leq 3$ and $\gamma := [\alpha_1, \alpha_2]$. The center $Z(G)$ is of order $p^2$, generated by $\alpha_4$ and $\gamma$. It is easy to check that $|\alpha^{G}| \leq p^3$ and 
 $|\alpha_1^{G}| \leq p^2.$ Therefore $|Aut_c(G| \leq p^5$. Let $H = \gen{\alpha_4}$. Since $\alpha_4 \in Z(G), H$ is normal. Consider the 
 quotient group $G/H$. One can check that the group $G/H$ belongs to the family $\Phi_6.$ Therefore it follows from \cite[Lemma 5.3]{MKY} that 
 $Aut_c(G/H) = Inn(G/H)$. Now define a map $\delta : \{\alpha, \alpha_1\} \rightarrow G$, such that $\delta(\alpha) = \alpha$ and $
 \delta(\alpha_1) = \alpha_2^{-1}\alpha_1\alpha_2 = \alpha_1\gamma$. Suppose that $|Aut_c(G| = p^5$. Then $\delta$ must extend to a class-preserving automorphism of
 $G$. It also induces a non-trivial class-preserving automorphism (say $\bar{\delta}$) of $G/H$. But then $\bar{\delta}$ is an inner automorphism of $G/H$ because $Aut_c(G/H) = Inn(G/H)$. 
 Now note that $\bar{\delta}$ is also a central automorphism of $G/H$. It follows that it
 is induced by some element in $Z_2(G/H) = \gen{\alpha_2H, Z(G/H)}$. Let 
 $\bar{\delta}$ be induced by $\alpha_2^tH$ for some $ t \in \mathbb{Z}$.
 Hence \[\bar{\delta}(\alpha H) = (\alpha_2^tH)^{-1}\alpha H (\alpha_2^tH) = \alpha\alpha_3^{-t}H .\]
 This must be equal to $\alpha H$. But then $t=0 \pmod{p}$ and hence $\bar{\delta}(\alpha_1H) = \alpha_1H$, a contradiction. Therefore we have $|Aut_c(G)| \leq p^4$. But 
 $|Inn(G)| = p^4$, therefore $Out_c(G) = 1$.   \hfill $\Box$

 \end{proof}

\begin{lemma}
 Let $G$ be the group $\Phi_{27}(1^6)$.  Then $Out_c(G) = 1$.
\end{lemma}

\begin{proof}
 The group $G$ is a  $p$-group of class $4$, minimally generated by $\alpha, \alpha_1, \beta$. The commutator subgroup $\gamma_2(G)$ is abelian and generated by $\alpha_{i+1} := [\alpha_i, \alpha]$ for $1 \leq i \leq 2$ and $\alpha_4 := [\alpha_3, \alpha] = [\alpha_1, \beta] = [\alpha_1, \alpha_2]$. The center $Z(G)$ is of order $p$, generated by $\alpha_4$. It is easy to check that 
 $|\alpha^{G}| \leq p^3, |\alpha_1^{G}| \leq p^2$ and $|\beta^{G}| = p$. Therefore $|Aut_c(G)| \leq p^6$. Define
 $\delta :\{\alpha, \alpha_1, \beta\} \rightarrow G$ such that $\delta(\alpha)= \alpha, \delta(\alpha_1) = \alpha_1$ and $\delta(\beta)= \alpha_1\beta\alpha_1^{-1} =
 \beta\alpha_4$. Suppose that $|Aut_c(G)| = p^6$. Then $\delta$ extends to a class-preserving automorphism of $G$. Also note that $\delta$ fixes $\alpha_2$, therefore
 $\delta(\alpha_2\beta^{-1}) = \alpha_2\beta^{-1}\alpha_4^{-1}$. Since $\delta$ is a class-preserving automorphism, there exist some $g \in  G$ such that
 $[\alpha_2\beta^{-1}, g] = \alpha_4^{-1}$. Note that $C_G(\alpha_2\beta^{-1}) = \gen{\alpha_1, \alpha_2, \alpha_3, \alpha_4, \beta}$. Therefore without loss of generality we can assume
 that $g = \alpha^{k_1}$. It can be calculated that 
 \[[\alpha_2\beta^{-1}, \alpha^{k_1}] = \alpha_3^{k_1}\alpha_4^{k_1(k_1-1)/2},\]
 which, for any value of $k_1$, can never be equal to $\alpha_4^{-1}$. Hence, we get a contradiction. Therefore $|Aut_c(G)| \leq p^5$. But 
 $|Inn(G)| = p^5$, therefore $Out_c(G) = 1$.    \hfill $\Box$

\end{proof}

\begin{lemma}
Let $G \in \{\Phi_{28}(222), \Phi_{29}(222)\}$. Then $|Out_c(G)| = 1$.
\end{lemma}

\begin{proof}
 The group $G$ is a $p$-group of class $4$, minimally generated by $\alpha$ and $\alpha_1$. The commutator subgroup $\gamma_2(G)$ is abelian and generated by $\alpha_{i+1} := [\alpha_i, \alpha]$ for $1 \leq i \leq 2$ and $\alpha_4 := [\alpha_3, \alpha] =  [\alpha_1, \alpha_2]$. The center $Z(G)$ is of order $p$, generated by $\alpha_4$. Note that $\gen{\alpha, \alpha_4} \leq C_G(\alpha)$  and hence $|C_G(\alpha)| \geq p^3$, therefore 
 $|\alpha^G| \leq p^3$. Now note that, since $\alpha_2^{(p)} = \alpha_4^y$, we have for $p =3, \alpha_2^p = \alpha_4^{y-1}$, and for $ p > 3, \alpha_2^p = \alpha_4^{y}$.
 Following is a routine calculation. 
 \[[\alpha_1, \alpha^p] = \alpha_2^p\alpha_3^{p(p-1)/2}\alpha_4^{\sum\limits_{n=1}^{p-2} n(n-1)/2} = \alpha_2^p\alpha_4^{(p-2)(p-1)p/3},\]
 which, for $p=3$, equals $\alpha_4^{y+1}$, and for $p > 3$, equals $\alpha_4^y$. But we have $[\alpha_1, \alpha_2^t] = \alpha_4^t$, therefore 
 $\alpha_2^{y+1}\alpha^{-p} \in C_G(\alpha_1)$ for $p=3$ and  $\alpha_2^{y}\alpha^{-p} \in C_G(\alpha_1)$ for $p>3$. Now it is easy to see that
 $|C_G(\alpha_1)| \geq p^4$. Hence $|\alpha_1^G| \leq p^2$. It follows from Lemma \ref{lem6} that $|Aut_c(G)| \leq p^5$. But $|G/Z(G)| = p^5$, therefore
  $Out_c(G) = 1$.    \hfill $\Box$

\end{proof}

\begin{lemma}
Let $G \in \{\Phi_{k}(1^6), \Phi_{34}(321)a \mid \ k = 31, \ldots, 33\}$. Then $Out_c(G) = 1$.
\end{lemma}

\begin{proof}
  The group $G$ is a  $p$-group of class $3$, minimally generated by $\alpha, \alpha_1$ and $\alpha_2$. 
 The commutator subgroup $\gamma_2(G)$ is abelian and generated by $\beta_i := [\alpha_i, \alpha]$ for $i = 1, 2$ and $\gamma := [\alpha_1, \beta_1]$.  For $k = 31, 32$, $[\alpha_2, \beta_2] = \gamma^y$, and for $k = 33$ and $\Phi_{34}(321)a$, $\gamma = [\beta_2, \alpha]$. The center $Z(G)$ is of order $p$, generated by $\gamma$. It is easy to see that, if $G \in \{\Phi_{k}(1^6)| \ k = 31, 32\}$, then
 $|\alpha_1^{G}| \leq p^2, |\alpha_2^{G}| \leq p^2, |\alpha^{G}| \leq p^2$, and if $G \in \{\Phi_{33}(1^6), \Phi_{34}(321)a\}$, then
 $|\alpha_1^{G}| \leq p^2, |\alpha_2^{G}| = p$ and $|\alpha^{G}| \leq p^3$.
 Therefore for all the four groups $G$, $|Aut_c(G)| \leq p^6$. Define a map $\delta : \{\alpha, \alpha_1, \alpha_2\} \rightarrow G$ such that
 $\alpha \mapsto \alpha, \alpha_1 \mapsto \alpha_1$ and $\alpha_2 \mapsto \alpha^{-1}\alpha_2\alpha = \alpha_2\beta_2$. Suppose that $|Aut_c(G)| = p^6$. Then $\delta$
 must extend to a class-preserving automorphism of $G$. Hence there exist elements $\eta_1$
 $\in \gamma_2(G)$ and $k_1, l_1, m_1 \in \mathbb{Z}$ such that 
 $[\alpha_1\alpha_2, \ \eta_1\alpha^{k_1}\alpha_1^{l_1}\alpha_2^{m_1}] = \beta_2$, but it is a routine calculation that
\[[\alpha_1\alpha_2, \ \eta_1\alpha^{k_1}\alpha_1^{l_1}\alpha_2^{m_1}] = \beta_1^{k_1}\beta_2^{k_1}\gamma^{a}\]
for some $a \in \mathbb{Z}$. Clearly it can not be equal to $\beta_2$ for any value of $k_1, l_1, m_1, r_1$ and $s_1$. 
Thus $\delta$ is not a class-preserving automorphism, and therefore  it follows that $|Aut_c(G)| \leq p^5$. But 
$|Inn(G)| = p^5$.  Hence $Aut_c(G) = Inn(G)$ proving that $Out_c(G) = 1$.   \hfill $\Box$

\end{proof}

 \begin{lemma}
 Let $G \in \{\Phi_{k}(1^6),  \Phi_{j}(222)a_0 \mid  k = 40, 41, \; j = 42,  43 \}$. Then $Out_c(G) = 1$.
\end{lemma}
\begin{proof}
 The group $G$ is a  $p$-group of class $4$, minimally generated by $\alpha_1$ and $\alpha_2$. The commutator subgroup $\gamma_2(G)$ is abelian and generated by 
 $\beta := [\alpha_1, \alpha_2], \beta_i := [\beta, \alpha_i]$ for $i = 1,2$ and $\gamma$,
 where, for $k = 40$, $\gamma := [\beta_1, \alpha_2] = [\beta_2, \alpha_1]$, for $k = 41$, $\gamma^{-\nu} := [\alpha_2, \beta_2] = [\alpha_1, \beta_1]^{-\nu}$, for $j = 42$, $\gamma := [\alpha_1, \beta_2] = [\alpha_2, \beta_1]$ and for $j = 43$, $\gamma^{-\nu} := [\alpha_2, \beta_2] = [\alpha_1, \beta_1]^{-\nu}$. The center $Z(G)$ is of order $p$, generated by $\gamma$. It is easy to see that $|\alpha_1^{G}| \leq p^3$ and $|\alpha_2^{G}| \leq p^3$. 
 Therefore $|Aut_c(G)| \leq p^6$. Define a map $\delta : \{\alpha_1, \alpha_2\} \rightarrow G$ such that
 $\alpha_1 \mapsto \alpha_1$ and $\alpha_2 \mapsto \alpha_2\beta_2$. Suppose that $|Aut_c(G)| = p^6$. Then $\delta$
 must extend to a class-preserving automorphism of $G$. Hence there exist  $\eta_1 (= \beta^{r_1}\beta_1^{s_1}\beta_2^{t_1}\gamma^{u_1}$ say) 
 $\in \gamma_2(G)$ and $k_1, l_1 \in \mathbb{Z}$ such that 
 $[\alpha_1\alpha_2, \ \eta_1\alpha_1^{k_1}\alpha_2^{l_1}] = \beta_2$. It is a routine calculation to show that,
 \begin{eqnarray*}
 [\alpha_1\alpha_2, \ \eta_1\alpha_1^{k_1}\alpha_2^{l_1}] &=& \beta^{l_1-k_1}\beta_1^{-k_1(k_1-1)/2 - r_1}\beta_2^{l_1(l_1 + 1)/2 - k_1l_1 - r_1}\gamma^{a}
 \end{eqnarray*}
for some $a \in \mathbb{Z}$. It is easy to see that if powers of $\beta$ and $\beta_1$ in the above expression are $0$ modulo $p$, then the power of $\beta_2$ is also $0$ modulo $p$. 
It follows that $\delta$ is not a class-preserving automorphism, and therefore $|Aut_c(G)| \leq p^5$. But $|Inn(G)| = p^5$, therefore $Out_c(G) = 1$.   \hfill $\Box$

\end{proof}

We are now  ready to prove the following theorem. 

\begin{thm} \label{prop1}
 Let $G$ be a group of order $p^6$ which belongs to one of the isoclinism family 
 $\Phi_k$ for $k = 2, \ldots,6,8,9,11,12,14,16, 17,19,23,25, \ldots,29, 31, \ldots,35,37, 40, \ldots,43,$ \cite[Section 4.6]{RJ}. Then $Out_c(G) = 1$.
\end{thm}

In the following proof, $\Phi_{k}(1^5)$ is the group of order $p^5$ from the isoclinism family $(k)$ of \cite[Section 4.5]{RJ} and $\Phi_{8}(32)$ is the group of order $p^5$ from the isoclinism family $(8)$ of \cite[Section 4.5]{RJ}.\\

\begin{proof}  
Note that $Aut_c(H \times K) \cong Aut_c(H) \times Aut_c(K)$, for any two groups $H$ and $K$.
Let $G \ \in \ \{\Phi_k(1^6), \ \Phi_8(321)a, \ | \ k = 2, \ldots,6,9\}$, then, since $\Phi_k(1^6) = \Phi_k(1^5) \times C_p$ and $\Phi_8(321)a = \Phi_8(32) \times C_p$, it follows from \cite[Theorem 5.5]{MKY} that  $Out_c(G) = 1$. Since the group
$\Phi_{12}(1^6)$ is a direct product of groups of order $p^3$, $Out_c(\Phi_{12}(1^6)) = 1$. The group $\Phi_{14}(1^6)$ is of nilpotency class 2 and $\gamma_2(\Phi_{14}(1^6))$ is cyclic, therefore from \cite[Corollary 3.6]{MKY}, we have $Out_c(\Phi_{14}(1^6)) = 1$.
Note that if $G \in \{\Phi_{16}(1^6),   \Phi_{k}(222) \mid k =  25, 26\}$, then $G$ is abelian by cyclic.  Hence by Lemma \ref{lem4}  $Out_c(G) = 1$. If $G \in \{\Phi_{k}(1^6) \mid k = 22, 35, 37\}$, then by Lemma \ref{lem6} it follows that $|Aut_c(G)| \le p^5$. Since $|Inn(G)| = p^5$, $Out_c(G) = 1$. This, along with lemmas 3.1 - 3.7, completes the proof of Theorem \ref{prop1}.        \hfill $\Box$

\end{proof}

\section{Groups $G$ with non-trivial $Out_c(G)$}
In this section we deal with those groups for which there exist a non-inner class-preserving automorphism.

\begin{lemma} \label{lem8}
 Let $G$ be the group $\Phi_{24}(1^6)$. 
Then $|Out_c(G)| = p$.
\end{lemma}

\begin{proof}
The group $G$ is a  $p$-group of class $4$, minimally generated by $\alpha, \alpha_1$ and $\beta$. The commutator subgroup $\gamma_2(G)$ is abelian and generated by $\alpha_{i+1} := [\alpha_i, \alpha]$ for $i = 1,2$ and $\alpha_4 := [\alpha_3, \alpha] = [\alpha_1, \beta]$. The center $Z(G)$ is of order $p$, generated by $\alpha_4$.  
We will show that for every element $g \in G-\gamma_2(G), Z(G) \leq [g, G]$. Let $g = \eta_1\alpha^{k_1}\alpha_1^{l_1}\beta^{m_1}$.
Then \[[g, \beta^{l_2}] = [\eta_1\alpha^{k_1}\alpha_1^{l_1}\beta^{m_1}, \beta^{l_2}] = \alpha_4^{l_1l_2}.\]
Therefore, if $l_1$ is non-zero modulo $p$, we have $Z(G) \leq [g, G]$. Let $l_1 \equiv 0\pmod{p}$. Then, since $\alpha_1^p \in \gamma_2(G)$, 
\[[g, \alpha_3^{k_2}] = [\eta_1\alpha^{k_1}\alpha_1^{l_1}\beta^{m_1}, \alpha_3^{k_2}] = \alpha_4^{k_1k_2}.\]
Therefore, if $k_1$ is non-zero modulo $p$, $Z(G) \leq [g, G]$. Now let $k_1 \equiv 0\pmod{p}$. Then
\[[g, \alpha_1^{m_2}] = [ \eta_1\beta^{m_1}, \alpha_1^{m_2}] = \alpha_4^{-m_1m_2}.\]
Therefore, if $m_1$ is non-zero modulo $p$, $Z(G) \leq [g, G]$.
It follows that for every $g \in G-\gamma_2(G), Z(G) \leq [g, G]$. Hence by lemmas \ref{lem2} and \ref{lem3} we have, 
\[|Aut_c(G)| \geq |Autcent(G)||G|/|Z_2(G)| = p^3p^6/p^3 = p^6.\] But, 
 $|\alpha^{G}| \leq p^3, |\alpha_1^{G}| \leq p^2$ and $|\beta^{G}| = p$. Therefore we have $|Aut_c(G)| \leq p^6$. Hence 
 $|Aut_c(G)| = p^6$. Since $|G/Z(G)| = p^5$, it follows that $|Out_c(G)| = p$.    \hfill $\Box$
 
 \end{proof}

\begin{lemma}
 Let $G$ be the group $\Phi_{30}(1^6)$. 
 Then $|Out_c(G)| = p$.
\end{lemma}

\begin{proof}
 The group $G$ is a  $p$-group of class $4$, minimally generated by $\alpha, \alpha_1$ and $\beta$. The commutator subgroup $\gamma_2(G)$ is abelian and generated by 
 $\alpha_2 := [\alpha_1, \alpha]$, $\alpha_3 := [\alpha_2, \alpha] = [\alpha_1, \beta]$ and $\alpha_4 := [\alpha_3, \alpha] = [\alpha_2, \beta]$. The center $Z(G)$ is of order $p$, generated by
 $\alpha_4$.  It is easy to see that $|\alpha^{G}| \leq p^3, |\alpha_1^{G}| \leq p^2$, 
 and $|\beta^{G}| \leq p^2$. Therefore $|Aut_c(G)| \leq p^7$. Define a map $\delta : \{\alpha, \alpha_1, \beta\} \rightarrow G$ such that
 $\alpha \mapsto \alpha\alpha_4, \alpha_1 \mapsto \alpha_1$ and $\beta \mapsto \beta\alpha_4$. By Lemma \ref{lemay}, the map $\delta$ extends to a central automorphism of $G$. 
 We will show that $\delta$ is also a non-inner class-preserving automorphism.
 Let $g = \eta_1\alpha^{k_1}\alpha_1^{l_1}\beta^{m_1}.$ Then 
 \[[g, \alpha_3^{k_2}] =  [\eta_1\alpha^{k_1}\alpha_1^{l_1}\beta^{m_1}, \alpha_3^{k_2}] = \alpha_4^{k_1k_2}.\] 
 Therefore, if $k_1$ is non-zero modulo $p$, we have  $Z(G) \leq  [g, G].$
 Hence $\delta$ maps $g$ to  a conjugate of $g$. Now let $k_1 \equiv 0\pmod{p}$, then 
 \[[g, \alpha_2^{m_2}] = [\eta_1\alpha_1^{l_1}\beta^{m_1}, \alpha_2^{m_2}] = \alpha_4^{-m_1m_2}.\]
 It follows that, if $m_1$ is non-zero modulo $p$, $\delta$ maps $g$ to a conjugate of $g$. Now let $m_1 \equiv 0\pmod{p}$. But then, $\delta(\eta_1\alpha_1^{l_1}) = \eta_1\alpha_1^{l_1},$ because being a central automorphism 
 $\delta$ fixes $\gamma_2(G)$ element-wise. We have shown that for every $g \in G$,
 $\delta(g)$ is a conjugate of $g$. Therefore $\delta$ is a class-preserving automorphism. Now suppose $\delta$ is an inner automorphism. Then being central, 
 $\delta$ is induced by some element in $Z_2(G) = \gen{\alpha_3, \alpha_4}$. Since $\alpha_4 \in Z(G)$, we can assume that $\delta$ is induced by $\alpha_3^t$. But then $\delta(\beta) = \beta$, 
 a contradiction. Therefore $\delta$ is non-inner class-preserving automorphism. Since $|Inn(G)| = p^5$, it follows that $|Aut_c(G)| \geq p^6.$
 Now define a map $\sigma : \{\alpha, \alpha_1, \beta\} \rightarrow G$ such that
 $\alpha \mapsto \alpha_1^{-1}\alpha\alpha_1 = \alpha\alpha_2^{-1}, \alpha_1 \mapsto \alpha_1$ and $\beta \mapsto \beta$. Suppose $|Aut_c(G)| = p^7$. Then 
 $\sigma$ extends to a class preserving automorphism. But then, $1 = \delta([\alpha, \beta]) = [\alpha\alpha_2^{-1}, \beta] = \alpha_4^{-1}$ which is not possible. Therefore, 
 we have $|Aut_c(G)| = p^6.$ Since $|G/Z(G)| = p^5$, we have $|Out_c(G)| = p$.     \hfill $\Box$

\end{proof}

\begin{lemma}
 Let $G$ be the group $\Phi_{36}(1^6).$   Then $|Out_c(G)| = p$.
\end{lemma}

\begin{proof}
 The group $G$ is a $p$-group of maximal class, minimally generated by $\alpha$ and $\alpha_1$. The commutator subgroup $\gamma_2(G)$ is abelian and generated by $\alpha_{i+1} := [\alpha_i, \alpha]$ for $i = 1, 2 ,3$ and $\alpha_5 := [\alpha_4, \alpha] = [\alpha_1, \alpha_2]$. The center  
 $Z(G)$ is of order $p$, generated by $\alpha_5.$ We will show that for every element $g \in G-\gamma_2(G), Z(G) \leq [g, G]$. Let $g = \eta_1\alpha^{k_1}\alpha_1^{l_1}$.
Then \[[g, \alpha_4^{k_2}] = [\eta_1\alpha^{k_1}\alpha_1^{l_1}, \alpha_4^{k_2}] = \alpha_4^{-k_1k_2}.\]
Therefore, if $k_1$ is non-zero modulo $p$, we have $Z(G) \leq [g, G]$. Let $k_1 \equiv 0\pmod{p}$. Then
\[[g, \alpha_2^{l_2}] = [\eta_1\alpha_1^{l_1}, \alpha_2^{l_2}] = \alpha_5^{l_1l_2}.\]
Therefore, if $l_1$ is non-zero modulo $p$, $Z(G) \leq [g, G]$. Let $l_1 \equiv 0\pmod{p}$, then we have $g \in \gamma_2(G)$ because $\alpha_1^p \in \gamma_2(G)$.
It follows that, for every $g \in G-\gamma_2(G), Z(G) \leq [g, G]$. Hence from lemmas \ref{lem2} and \ref{lem3} we get  
\[|Aut_c(G)| \geq |Autcent(G)||G|/|Z_2(G)| = p^2p^6/p^2 = p^6.\] But it is easy to see that 
 $|\alpha^{G}| \leq p^4$ and $|\alpha_1^{G}| \leq p^2$. Therefore $|Aut_c(G)| \leq p^6$. Hence 
 $|Aut_c(G)| = p^6$. Since $|G/Z(G)| = p^5$, we have  $|Out_c(G)| = p$.    \hfill $\Box$

 \end{proof}

\begin{lemma}
 Let $G$ be the group $\Phi_{38}(1^6)$.  Then $|Out_c(G)| = p$.
\end{lemma}

\begin{proof}
The group $G$ is a $p$-group of maximal class, minimally generated by $\alpha$, $\alpha_1$. The commutator subgroup $\gamma_2(G)$ is abelian and generated by $\alpha_{i+1} := [\alpha_i, \alpha]$ for $i = 1, 2 ,3$ and $\alpha_5 := [\alpha_4, \alpha] = [\alpha_1, \alpha_3]$. Also $[\alpha_1, \alpha_2] = \alpha_4\alpha_5^{-1}.$ The center $Z(G)$ is generated by $\alpha_5.$
 We show that for every element $g \in G-\gamma_2(G), Z(G) \leq [g, G]$. Let $g = \eta_1\alpha^{k_1}\alpha_1^{l_1}$.
Then \[[g, \alpha_4^{k_2}] = [\eta_1\alpha^{k_1}\alpha_1^{l_1}, \alpha_4^{k_2}] = \alpha_5^{-k_1k_2}.\]
Therefore, if $k_1$ is non-zero modulo $p$, $Z(G) \leq [g, G]$. Hence, let $k_1 \equiv 0\pmod{p}$. Then
\[[g, \alpha_3^{l_2}] = [\eta_1\alpha_1^{l_1}, \alpha_3^{l_2}] = \alpha_5^{l_1l_2}.\]
Therefore, if $l_1$ is non-zero modulo $p$, we have $Z(G) \leq [g, G]$. If $l_1 \equiv 0\pmod{p}$, then we have $g \in \gamma_2(G)$ because $\alpha_1^p \in \gamma_2(G)$. 
It follows that, for every $g \in G-\gamma_2(G), \ Z(G) \leq [g, G]$. Therefore applying lemmas \ref{lem2} and \ref{lem3} we get 
\[|Aut_c(G)| \geq |Autcent(G)||G|/|Z_2(G)| = p^2p^6/p^2 = p^6.\] It is easy to check that  
 $|\alpha^{G}| \leq p^4$ and $|\alpha_1^{G}| \leq p^3$.  Hence the upper bound for  $|Aut_c(G)|$ is $p^7$. Suppose that $|Aut_c(G)| = p^7$. 
 Then the map $\delta$ defined on the generating set  $\{\alpha, \alpha_1\}$ as $\delta(\alpha) = \alpha$ and $\delta(\alpha_1) = \alpha_2^{-1}\alpha_1\alpha_2 = \alpha_1\alpha_4\alpha_5^{-1}$
 must extend to a class-preserving automorphism automorphism of $G$. Note that \[\delta(\alpha_2) = \delta([\alpha_1, \alpha]) = [\alpha_1\alpha_4\alpha_5^{-1}, \alpha] = \alpha_2\alpha_5.\]
 Since $\delta$ is a class-preserving automorphism, there exists an $\eta_1 \in \gamma_2(G)$ and $k_1, l_1 \in \mathbb{Z}$ such that
 \[[\alpha_2, \eta_1\alpha^{k_1}\alpha_1^{l_1}] = \alpha_5.\]
 We have that
 \[[\alpha_2, \eta_1\alpha^{k_1}\alpha_1^{l_1}] = \alpha_3^{k_1}\alpha_4^{k_1(k_1-1)/2 - l_1}\alpha_5^{l_1 - k_1l_1 + \sum\limits_{n=1}^{k_1-2} n(n-1)/2},\]
 which for no values of $k$ and $k_1$ can be equal to $\alpha_5$. This gives a  contradiction. Hence $|Aut_c(G)| = p^6$.
 Since $|G/Z(G)| = p^5$, we have  $|Out_c(G)| = p$.    \hfill $\Box$

 \end{proof}

\begin{lemma} \label{lem9}
 Let $G$ be the group $\Phi_{39}(1^6)$. 
 Then $|Out_c(G)| = p$.
\end{lemma}

\begin{proof}
  The group $G$ is a $p$-group of maximal class, minimally generated by $\alpha, \alpha_1$. The commutator subgroup $\gamma_2(G)$ is generated by $\alpha_{i+1} := [\alpha_i, \alpha]$ for $i = 1,2$, $\alpha_4 := [\alpha_3, \alpha] = [\alpha_1, \alpha_2]$, and $\alpha_5 := [\alpha_2, \alpha_3] = [\alpha_3, \alpha_1] = [\alpha_4, \alpha_1]$. Note that any element of $\gamma_2(G)$
 can be written as $\alpha_2^r\alpha_3^s\alpha_4^t\alpha_5^u$ for some $r, s, t, u \in \mathbb{Z}$. The center $Z(G)$ is generated by $\alpha_5$. It is easy to see that
 $|\alpha^{G}| \leq p^3$ and $|\alpha_1^{G}| \leq p^3$. 
  Therefore $|Aut_c(G)| \leq p^6$. Define a map $\delta : \{\alpha, \alpha_1 \} \rightarrow G$ such that
 $\alpha \mapsto \alpha$, and $\alpha_1 \mapsto \alpha_2^{-1}\alpha_1\alpha_2 = \alpha_1\alpha_4$. 
 We will show that $\delta$ extends to a class preserving automorphism of $G$. It is easy to check that $\delta$ preserves all the defining relations of the group. Hence $\delta$ extends to an endomorphism. Note that $\delta
 $ fixes $\gamma_2(G)$ element-wise, therefore for $k_1, l_1 \in \mathbb{Z}$ and $\eta_1 =  \alpha_2^{r_1}\alpha_3^{s_1}\alpha_4^{t_1}\alpha_5^{u_1} 
 \in \gamma_2(G)$, 
 \[\delta(\eta_1\alpha^{k_1}\alpha_1^{l_1}) = \eta_1\alpha^{k_1}\alpha_1^{l_1}\alpha_4^{l_1}\alpha_5^{l_1(l_1-1)/2}.\]
 It is a routine calculation that
 \[[\eta_1\alpha^{k_1}\alpha_1^{l_1}, \alpha_3^{s_2}\alpha_4^{t_2}] = \alpha_5^{r_1s_2 - l_1s_2 - l_1t_2 - l_1k_1s_2}\alpha_4^{-k_1s_2}.\]
 Since $\delta(\eta_1\alpha^{k_1}) = \eta_1\alpha^{k_1}$, which is obviously a conjugate of $\eta_1\alpha^{k_1}$, let $l_1$ be non-zero modulo $p$. 
 Now if $k_1$ is non-zero modulo $p$, then clearly there exist $s_2$ and $t_2$ such that
 \[-k_1s_2 \equiv l_1  \pmod{p}\]
 and
 \[r_1s_2 - l_1s_2 - l_1t_2 - l_1k_1s_2 \equiv l_1(l_1-1)/2 \pmod{p}.\]
 Suppose $k_1 \equiv 0\pmod{p}$, then it can be calculated that
 \[[\eta_1\alpha_1^{l_1}, \alpha_2^{r_2}\alpha_4^{t_2}] = \alpha_5^{-s_1r_2 + r_2l_1(l_1-1)/2 - l_1t_2}\alpha_4^{l_1r_2}.\]
 Since $l_1$ is non-zero modulo $p$, clearly there exist $r_2$ and $t_2$ such that
 \[l_1r_2 \equiv l_1 \pmod{p}\]
 and
 \[-s_1r_2 + r_2l_1(l_1-1)/2 - l_1t_2 \equiv l_1(l_1-1)/2  \pmod{p}.\]
 It follows that $\delta$ maps every element of $G$ to a conjugate of itself. Therefore $\delta$ is a bijection, and hence a class-preserving automorphism of $G$. Now we show that
 $\delta$ is a non-inner automorphism. On the contrary, suppose that $\delta$ is an inner automorphism. Note that $g^{-1}\delta(g) \in Z_2(G)$. Thus it follows that $\delta$ is induced by some element in 
 $Z_3(G) = \gen{\alpha_3, \alpha_4, Z(G)}$, where $Z_3(G)$ denotes the third term in the upper central series of $G$. Let it be induced by $\alpha_3^{s_1}\alpha_4^{t_1}$. Since $\delta(\alpha) = \alpha$, we have
 $\alpha_3^{-s_1}\alpha\alpha_3^{s_1} = \alpha$, which implies that $s_1 \equiv 0\pmod{p}$.  But then $\delta(\alpha_1) = \alpha_4^{-t_1}\alpha_1\alpha_4^{t_1} = \alpha_1\alpha_5^{-t_1}$, 
 a contradiction. Therefore $\delta$ is a non-inner class preserving automorphism. Since $|Inn(G)| = p^5$, we have $|Aut_cG)| = p^6$, and therefore $|Out_c(G)| = p$.    \hfill $\Box$
 
\end{proof}
 \begin{lemma}\label{lem10}
 Let $G$ be the group $\Phi_{13}(1^6)$.  Then $|Out_c(G)| = p^2$.
\end{lemma}

\begin{proof}
 The group $G$ is a special $p$-group, minimally generated by $\alpha_1, \alpha_2, \alpha_3$ and $\alpha_4$.
The commutator subgroup $\gamma_2(G)$ is generated by $\beta_1 := [\alpha_1, \alpha_2]$ and $\beta_2 := [\alpha_1, \alpha_3] = [\alpha_2, \alpha_4]$. The conjugates of $\alpha_1, \alpha_2, \alpha_3$ and $\alpha_4$ are , 
 $\alpha_1\beta_1^r\beta_2^s$, $\alpha_2\beta_1^r\beta_2^s, \alpha_3\beta_2^s$ and $\alpha_4\beta_2^s$ respectively where $r$ and $s$ varies over $\mathbb{Z}$. 
 Since the exponent of $\gamma_2(G)$ is $p$, it follows that $|\alpha_1^G| = |\alpha_2^G| = p^2$ and $|\alpha_3^G| = |\alpha_4^G| = p$. 
 Therefore by Lemma \ref{lem6}, $|Aut_c(G)| \leq p^6$.
 
 Define a map $\delta : \{\alpha_1, \alpha_2, \alpha_3, \alpha_4\} \rightarrow G$ such that $\alpha_1 \mapsto \alpha_1\beta_1^{r_1}\beta_2^{s_1}, \alpha_2 \mapsto 
 \alpha_2\beta_1^{r_2}\beta_2^{s_2}$, $\alpha_3 \mapsto  \alpha_3\beta_2^{s_3}$ and  $\alpha_4 \mapsto \alpha_4\beta_2^{s_4}$, for some $r_1, s_1, r_2, s_2, s_3, s_4 \in \mathbb{Z}$. By Lemma \ref{lemay}
this map extends to a central automorphism of $G$. Since $\delta$ fixes $\gamma_2(G)$ element-wise,
 for $k_1, l_1, m_1, n_1 \in \mathbb{Z}$ and $\eta_1 \in  \gamma_2(G)$, 
 \[\delta(\eta_1\alpha_1^{k_1}\alpha_2^{l_1} \alpha_3^{m_1}\alpha_4^{n_1}) = 
 \eta\alpha_1^{k_1}\alpha_2^{l_1} \alpha_3^{m_1}\alpha_4^{n_1}\beta_1^{k_1r_1 + l_1r_2}\beta_2^{k_1s_1 + l_1s_2 + m_1s_3 + n_1s_4}.\]
Therefore $\delta$ extends to a class-preserving automorphism if and only if for every $k_1, l_1, m_1, n_1 \in \mathbb{Z}$, and $\eta_1 \in \gamma_2(G)$, 
 there exist $k_2, l_2, m_2, n_2$ (depending on $k_1, l_1, m_1, n_1$) and $\eta_2 \in \gamma_2(G)$ such that
 \[[\eta_1\alpha_1^{k_1}\alpha_2^{l_1}\alpha_3^{m_1}\alpha_4^{n_1}, \eta_2\alpha_1^{k_2}\alpha_2^{l_2}\alpha_3^{m_2}\alpha_4^{n_2}] = 
   \beta_1^{k_1r_1 + l_1r_2}\beta_2^{k_1s_1 + l_1s_2 + m_1s_3 + n_1s_4}.\]
 Expanding the left hand side, we get
\[\beta_1^{k_1l_2 - k_2l_1}\beta_2^{k_1m_2 - k_2m_1 + l_1n_2 - l_2n_1} = \beta_1^{k_1r_1 + l_1r_2}\beta_2^{k_1s_1 + l_1s_2 + m_1s_3 + n_1s_4}.\]
Comparing the powers of $\beta_i$'s, we see that $\delta$ extends to a class-preserving automorphism if the following equations hold:
\[k_1l_2 - k_2l_1 \equiv k_1r_1 + l_1r_2 \pmod{p},\]
\[k_1m_2 - k_2m_1 + l_1n_2 - l_2n_1 \equiv k_1s_1 + l_1s_2 + m_1s_3 + n_1s_4 \pmod{p}.\]
It is easy to see that for any given $k_1, l_1, m_1, n_1$ there exist $k_2, l_2, m_2, n_2$ such that the above two equations are satisfied. Thus it follows that 
$|Aut_c(G)| = p^6$. Since $|G/Z(G)| = p^4, |Out_c(G)| = p^2$.    \hfill $\Box$

\end{proof}

\begin{lemma}
 Let $G$ be the group $\Phi_{18}(1^6)$.   Then $|Out_c(G)| = p^2$.
\end{lemma}

\begin{proof}
 The group $G$ is a  $p$-group of class $3$, minimally generated by $\alpha, \alpha_1, \beta$. 
 The commutator subgroup $\gamma_2(G)$ is abelian and generated by $\alpha_2:= [\alpha_1, \alpha], \alpha_3 := [\alpha_2, \alpha] = [\alpha_1, \beta]$ and $\gamma := [\alpha, \beta]$. The center $Z(G)$ is of order $p^2$, generated by $\alpha_3$ and $\gamma$.
 Note that $|\alpha^{G}| \leq p^3, |\alpha_1^{G}| \leq p^2$ and $|\beta^{G}| \leq p^2$. It follows from Lemma \ref{lem6} that 
 $|Aut_c(G)| \leq p^7$.
 Now define a map $\delta : \{\alpha, \alpha_1, \beta \} \rightarrow G$ such that $\alpha \mapsto \alpha\alpha_3^{r_1}\gamma^{s_1}, \alpha_1 \mapsto 
 \alpha_1\alpha_3^{r_2}$ and $\beta \mapsto  \beta\alpha_3^{r_3}$ for some $r_1, s_1, r_2, r_3 \in \mathbb{Z}$. 
 By Lemma \ref{lemay} this map extends to a central automorphism of $G$.
 Let $g = \eta_1\alpha^{k_1}\alpha_1^{l_1}\beta^{m_1}$, where $\eta_1 \in \gamma_2(G)$ and
 $k_1, l_1, m_1 \in \mathbb{Z}$. Let $k_1 \equiv 0\pmod{p}$.  Then note that $\delta(g) = g\alpha_3^r$ for some $r \in \mathbb{Z}$ and 
 \[[g, \beta^{l_2}] = [\eta_1\alpha_1^{l_1}\beta^{m_1}, \beta^{l_2}] = \alpha_3^{l_1l_2}.\]
 Therefore if $l_1$ is non-zero modulo $p$, we have $\gen{\alpha_3} \leq [g, G]$. Let $l_1 \equiv 0\pmod{p}$. Note that $\alpha_1^p \in Z(G)$, hence 
 \[[g, \alpha_1^{m_2}] = [\eta_1\beta^{m_1}, \alpha_1^{m_2}] = \alpha_3^{-m_1m_2}.\]
 Therefore if $m_1$ is non-zero modulo $p$, we have $\gen{\alpha_3} \leq [g, G]$. Thus we have shown that, if $k_1 \equiv 0\pmod{p}$, then  
 $g^{-1}\delta(g) \in [g, G].$ It follows that, if $k_1 \equiv 0\pmod{p}$,  $\delta$ maps $g$ to a conjugate of $g$. Now suppose $k_1$ is non-zero modulo $p$.  Then
 \[[g, \alpha_2^{k_2}] = [\eta_1\alpha^{k_1}\alpha_1^{l_1}\beta^{m_1}, \alpha_2^{k_2}] = \alpha_3^{-k_1k_2},\] 
 so that $\gen{\alpha_3} \leq [g, Z_2(G)]$. Also, we have
 \[[g, \beta^{n_2}] = [\eta_1\alpha^{k_1}\alpha_1^{l_1}\beta^{m_1}, \beta^{n_2}] = \alpha_3^{l_1n_2}\gamma^{k_1n_2}.\]
 Since $\beta \in Z_2(G)$ and $[g, Z_2(G)]$ is a subgroup, it follows that $Z(G) \leq [g, G]$. Because 
 $\delta$ is a central automorphism, we have $g^{-1}\delta(g) \in [g, G]$. We have shown that, for every $g \in G$,  
 $g^{-1}\delta(g) \in [g, G].$ 
 It follows that $\delta$ is a class-preserving automorphism. 
 Since $r_1, s_1, r_2, r_3$ were arbitrary, we have that $|Aut_c(G) \cap Autcent(G)| \geq p^4$.
 Applying Lemma \ref{lempkr} we get $|Z(Aut_c(G))| \geq p^4$. Note that $Aut_c(G)$ is non-abelian because $G$ is a group of  class 3. 
 Therefore $|Aut_c(G)| \geq p^6$. Now suppose that $|Aut_c(G)| = p^7.$ Then the map $\sigma$ defined on the generating set $\{\alpha, \alpha_1, \beta\}$ 
 as $\alpha \mapsto \alpha, \alpha_1 \mapsto \alpha_1$ and $\beta \mapsto  \alpha\beta\alpha^{-1} = \beta\gamma$ extends to a class-preserving automorphism. 
 Hence there exist  $\eta_2 \in \gamma_2(G)$ and $k_2, l_2, m_2 \in \mathbb{Z}$ such that 
 $[\alpha_1\beta, \ \eta_2\alpha^{k_2}\alpha_1^{l_2}\beta^{m_2}] = \gamma$, but by a routine calculation it can be checked that
 \[[\alpha_1\beta, \ \eta_2\alpha^{k_2}\alpha_1^{l_2}\beta^{m_2}] = \alpha_2^{k_2}\beta\alpha_3^{m_2 - l_2 + k_2(k_2 - 1)/2}\gamma^{-k_2},\]
which can not be equal to $\gamma$ for any values of $k_2, l_2$ and $m_2$. Therefore we get a contradiction and it follows that  $|Aut_c(G)| = p^6$. Hence $|Out_c(G)| = p^2$ as $|Inn(G)| = p^4$.  \hfill $\Box$

 \end{proof}
 
\begin{lemma} \label{lem11}
 Let $G$ be the group $\Phi_{20}(1^6)$.   Then $|Out_c(G)| = p^2$.
\end{lemma}

\begin{proof}
The group $G$ is a  $p$-group of class $3$, minimally generated by $\alpha, \alpha_1, \alpha_2$. 
 The commutator subgroup $\gamma_2(G)$ is abelian and generated by $\beta:= [\alpha_1, \alpha_2], \beta_1 := [\beta, \alpha_1]$ and $\beta_2 := [\beta, \alpha_2] = [\alpha, \alpha_1]$. The center $Z(G)$ is of order $p^2$, generated by $\beta_1$ and $\beta_2$. 
 Note that $|\alpha^{G}| = p, |\alpha_1^{G}| \leq p^3$ and $|\alpha_2^{G}| \leq p^2$. It follows from Lemma \ref{lem6} that 
 $|Aut_c(G)| \leq p^6$.
 Now define a map $\delta : \{\alpha, \alpha_1, \alpha_2\} \rightarrow G$ such that $\alpha \mapsto \alpha\beta_2^{t_1}, \alpha_1 \mapsto 
 \alpha_1\beta_1^{s_2}\beta_2^{t_2}$ and $\alpha_2 \mapsto  \alpha_2\beta_2^{t_3}$ for some $t_1, s_2, t_2, t_3 \in \mathbb{Z}$. 
 By Lemma \ref{lemay} this map extends to a central automorphism of $G$. 
 Let $g = \eta_1\alpha^{k_1}\alpha_1^{l_1}\alpha_2^{m_1}$, where $\eta_1 = \beta^{u_1}\beta_1^{v_1}\beta_2^{w_1}$ and
 $k_1, l_1, m_1, u_1, v_1, w_1 \in \mathbb{Z}$. Note that, if $l_1 \equiv 0\pmod{p}$, then  $\delta(g) = g\beta_2^r$ for some $r \in \mathbb{Z}$. 
 Consider \[[g, \beta^{m_2}] = [\eta_1\alpha^{k_1}\alpha_2^{m_1}, \beta^{m_2}] = \beta_2^{-m_1m_2}.\]
 Therefore if $m_1$ is non-zero modulo $p$, we have $\gen{\beta_2} \leq [g, G]$. Let $m_1 \equiv 0\pmod{p}$. Note that $\alpha_2^p \in Z(G)$, hence 
 \[[g, \alpha_2^{u_2}] = [\eta_1\alpha^{k_1}, \alpha_2^{u_2}] = [\beta^{u_1}\beta_1^{v_1}\beta_2^{w_1}\alpha^{k_1}, \ \alpha_2^{u_2}] =   \beta_2^{u_1u_2}.\]
 Therefore if $u_1$ is non-zero modulo $p$, we have $\gen{\beta_2} \leq [g, G]$. Let $u_1 \equiv 0\pmod{p}$. Then
 \[[g, \alpha_1^{k_2}] = [\alpha^{k_1}, \alpha_1^{k_2}] = \beta_2^{k_1k_2},\]
 so that if $k_1$ is non-zero modulo $p$, then $\gen{\beta} \leq [g, G]$. Thus we have shown that if $l_1 \equiv 0\pmod{p}$,
 then $g^{-1}\delta(g) \in [g, G].$ It follows that if $l_1 \equiv 0\pmod{p}$, then $\delta$ maps $g$ to a conjugate of $g$. Now suppose that $l_1$ is non-zero modulo $p$. Then
 \[[g, \alpha^{l_2}] = [\eta_1\alpha^{k_1}\alpha_1^{l_1}\alpha_2^{m_1}, \alpha^{l_2}] = \beta_2^{-l_1l_2},\] 
 so that $\gen{\beta_2} \leq [g, Z_2(G)]$. Also we have
 \[[g, \beta^{n_2}] = [\eta_1\alpha^{k_1}\alpha_1^{l_1}\alpha_2^{m_1}, \beta^{n_2}] = \beta_1^{-l_1n_2}\beta_2^{-m_1n_2}.\]
 Since $\beta \in Z_2(G)$ and $[g, Z_2(G)]$ is a subgroup, it follows that $Z(G) \leq [g, G]$. Because 
 $\delta$ is a central automorphism,  we have $g^{-1}\delta(g) \in [g, G]$. It follows that $\delta$ is a class-preserving automorphism of $G$. 
 Since $t_1, s_2, t_2, t_3$ were arbitrary, we have that $|Aut_c(G) \cap Autcent(G)| \geq p^4$.
 Applying Lemma \ref{lempkr} we get $|Z(Aut_c(G))| \geq p^4$. But $Aut_c(G)$ is non-abelian since $G$ is of class 3.
 Therefore $|Aut_c(G)| \geq p^6$. Hence $|Aut_c(G)| = p^6.$ Since $|G/Z(G)| = p^4$, we have $|Out_c(G)| = p^2$.    \hfill $\Box$

\end{proof}

Now we are ready to prove the following result, which together with Theorem \ref{prop1} proves  Theorem A. 

\begin{thm} \label{prop2}
 Let $G$ be a group of order $p^6$. Then the following hold true.
 \begin{enumerate}
 \item  If $G$ belongs to any of the isoclinism families $\Phi_k$ for $k = 7, 10, 24, 30, 36, 38, 39$, then $|Out_c(G)| = p$.  
 \item  If $G$ belongs to any of the isoclinism families $\Phi_k$ for $k = 13, 18, 20$, then $|Out_c(G)| = p^2$.
 \item  If $G$ belongs to any of the isoclinism families $\Phi_k$ for $k = 15, 21$, then $|Out_c(G)| = p^4$.
 \end{enumerate}
\end{thm}

\begin{proof}
 Let $G$ be either the group $\Phi_{7}(1^6)$ or the group $\Phi_{10}(1^6)$. Then, since 
$\Phi_{7}(1^6) = \Phi_{7}(1^5) \times C_p$ and $\Phi_{10}(1^6) = \Phi_{10}(1^5) \times C_p$, it follows from \cite[Lemma 5.1 and Lemma 5.2]{MKY} that  $|Out_c(G)| = p$.
With this observation in hands, (i) follows from lemmas \ref{lem8}-\ref{lem9}. It is readily seen that (ii) follows from lemmas \ref{lem10}-\ref{lem11}. Now
let $G$ be the group $\Phi_{15}(1^6)$. We observe that
James' list of groups of order $p^6$ and class 2 consists of exactly 5 isoclinism families $\Phi_{k}$ for $k =11, \ldots, 15$. As we have seen above  that
$|Aut_c(H)| \neq p^8$ for $H \in \{\Phi_k \mid 11\le k \le 14\}$. But, as shown by Burnside \cite{WB}, there exists a group $W$ of order $p^6$ and of nilpotency class 2 such that $|Aut_c(W)| = p^8$. 
Therefore $G$ must be isoclinic to the group $W$ and $|Aut_c(G)| = p^8$. Since $|G/Z(G)| = p^4$, we have $|Out_c(G)| = p^4$. 
Next, it follows from \cite[Proposition 5.8]{MKYLMS}, that $|Out_c(\Phi_{21}(1^6))| = p^4$. This completes the proof of the theorem.  \hfill $\Box$

\end{proof}

\noindent{\bf Acknowledgements.}  We thank the referee for his/her useful comments and suggestions which make the paper more readable.

\end{document}